\documentclass[12pt]{amsart}
\usepackage{graphics}
\usepackage{amsfonts,amssymb,color}
\usepackage[mathscr]{eucal}
\usepackage{amsmath, amsthm}
\usepackage{mathrsfs}
\usepackage{amsbsy}
\usepackage{dsfont}
\usepackage{bbm}
\usepackage{wasysym}
\usepackage{stmaryrd}
\usepackage{url}

\input xypic
\xyoption{all}

\makeindex
\makeglossary

\begin{document}
\baselineskip = 16pt

\newcommand \ZZ {{\mathbb Z}}
\newcommand \NN {{\mathbb N}}
\newcommand \RR {{\mathbb R}}
\newcommand \PR {{\mathbb P}}
\newcommand \AF {{\mathbb A}}
\newcommand \GG {{\mathbb G}}
\newcommand \QQ {{\mathbb Q}}
\newcommand \CC {{\mathbb C}}
\newcommand \bcA {{\mathscr A}}
\newcommand \bcC {{\mathscr C}}
\newcommand \bcD {{\mathscr D}}
\newcommand \bcF {{\mathscr F}}
\newcommand \bcG {{\mathscr G}}
\newcommand \bcH {{\mathscr H}}
\newcommand \bcM {{\mathscr M}}
\newcommand \bcJ {{\mathscr J}}
\newcommand \bcK {{\mathscr K}}
\newcommand \bcL {{\mathscr L}}
\newcommand \bcO {{\mathscr O}}
\newcommand \bcP {{\mathscr P}}
\newcommand \bcQ {{\mathscr Q}}
\newcommand \bcR {{\mathscr R}}
\newcommand \bcS {{\mathscr S}}
\newcommand \bcV {{\mathscr V}}
\newcommand \bcW {{\mathscr W}}
\newcommand \bcX {{\mathscr X}}
\newcommand \bcY {{\mathscr Y}}
\newcommand \bcZ {{\mathscr Z}}
\newcommand \goa {{\mathfrak a}}
\newcommand \gob {{\mathfrak b}}
\newcommand \goc {{\mathfrak c}}
\newcommand \gom {{\mathfrak m}}
\newcommand \gon {{\mathfrak n}}
\newcommand \gop {{\mathfrak p}}
\newcommand \goq {{\mathfrak q}}
\newcommand \goQ {{\mathfrak Q}}
\newcommand \goP {{\mathfrak P}}
\newcommand \goM {{\mathfrak M}}
\newcommand \goN {{\mathfrak N}}
\newcommand \uno {{\mathbbm 1}}
\newcommand \Le {{\mathbbm L}}
\newcommand \Spec {{\rm {Spec}}}
\newcommand \Gr {{\rm {Gr}}}
\newcommand \Pic {{\rm {Pic}}}
\newcommand \Jac {{{J}}}
\newcommand \Alb {{\rm {Alb}}}
\newcommand \Corr {{Corr}}
\newcommand \Chow {{\mathscr C}}
\newcommand \Sym {{\rm {Sym}}}
\newcommand \Prym {{\rm {Prym}}}
\newcommand \cha {{\rm {char}}}
\newcommand \eff {{\rm {eff}}}
\newcommand \tr {{\rm {tr}}}
\newcommand \Tr {{\rm {Tr}}}
\newcommand \pr {{\rm {pr}}}
\newcommand \ev {{\it {ev}}}
\newcommand \cl {{\rm {cl}}}
\newcommand \interior {{\rm {Int}}}
\newcommand \sep {{\rm {sep}}}
\newcommand \td {{\rm {tdeg}}}
\newcommand \alg {{\rm {alg}}}
\newcommand \im {{\rm im}}
\newcommand \gr {{\rm {gr}}}
\newcommand \op {{\rm op}}
\newcommand \Hom {{\rm Hom}}
\newcommand \Hilb {{\rm Hilb}}
\newcommand \Sch {{\mathscr S\! }{\it ch}}
\newcommand \cHilb {{\mathscr H\! }{\it ilb}}
\newcommand \cHom {{\mathscr H\! }{\it om}}
\newcommand \colim {{{\rm colim}\, }} 
\newcommand \End {{\rm {End}}}
\newcommand \coker {{\rm {coker}}}
\newcommand \id {{\rm {id}}}
\newcommand \van {{\rm {van}}}
\newcommand \spc {{\rm {sp}}}
\newcommand \Ob {{\rm Ob}}
\newcommand \Aut {{\rm Aut}}
\newcommand \cor {{\rm {cor}}}
\newcommand \Cor {{\it {Corr}}}
\newcommand \res {{\rm {res}}}
\newcommand \red {{\rm{red}}}
\newcommand \Gal {{\rm {Gal}}}
\newcommand \PGL {{\rm {PGL}}}
\newcommand \Bl {{\rm {Bl}}}
\newcommand \Sing {{\rm {Sing}}}
\newcommand \spn {{\rm {span}}}
\newcommand \Nm {{\rm {Nm}}}
\newcommand \inv {{\rm {inv}}}
\newcommand \codim {{\rm {codim}}}
\newcommand \Div{{\rm{Div}}}
\newcommand \sg {{\Sigma }}
\newcommand \DM {{\sf DM}}
\newcommand \Gm {{{\mathbb G}_{\rm m}}}
\newcommand \tame {\rm {tame }}
\newcommand \znak {{\natural }}
\newcommand \lra {\longrightarrow}
\newcommand \hra {\hookrightarrow}
\newcommand \rra {\rightrightarrows}
\newcommand \ord {{\rm {ord}}}
\newcommand \Rat {{\mathscr Rat}}
\newcommand \rd {{\rm {red}}}
\newcommand \bSpec {{\bf {Spec}}}
\newcommand \Proj {{\rm {Proj}}}
\newcommand \pdiv {{\rm {div}}}
\newcommand \Supp {{\rm {Supp}}}
\newcommand \CH {{\it {CH}}}
\newcommand \NS {{\it {NS}}}
\newcommand \wt {\widetilde }
\newcommand \ac {\acute }
\newcommand \ch {\check }
\newcommand \ol {\overline }
\newcommand \Th {\Theta}
\newcommand \cAb {{\mathscr A\! }{\it b}}

\newenvironment{pf}{\par\noindent{\em Proof}.}{\hfill\framebox(6,6)
\par\medskip}

\newtheorem{theorem}[subsection]{Theorem}
\newtheorem{conjecture}[subsection]{Conjecture}
\newtheorem{proposition}[subsection]{Proposition}
\newtheorem{lemma}[subsection]{Lemma}
\newtheorem{remark}[subsection]{Remark}
\newtheorem{remarks}[subsection]{Remarks}
\newtheorem{definition}[subsection]{Definition}
\newtheorem{corollary}[subsection]{Corollary}
\newtheorem{example}[subsection]{Example}
\newtheorem{examples}[subsection]{examples}
\title{Algebraic cycles on hyperplane sections of  hypersurfaces in $\PR^n$ for $n=5,6$}
\author{Kalyan Banerjee}
\begin{abstract}
Let $X$ be a cubic hypersurface in $\PR^6$ or a hypersurface of degree greater than equal to $7$ in $\PR^5$. In this note we try to understand, for a very general hyperplane section of $X$, the non-injectivity locus of the corresponding push-forward homomorphism at the level of Chow group of certain dimension.
\end{abstract}
\maketitle

\section{Introduction}
The question of injectivity at the level of algebraically trivial one cycles on a hypersurface in $\PR^n$ is an interesting one. The conjecture due to Nori and Paranjape says that for a hypersurface of degree greater  than $2n-3$ in $\PR^n$, the group of algebraically trivial  cycles (of certain codimension) modulo rational equivalence is isomorphic to $\ZZ$, see, \cite{Nori},\cite{Paranjape} for the precise formulation of the conjecture. Inspired by these conjectures, the authors in \cite{BIL} has formulated the following question:

Suppose $X$ is a smooth projective variety and $D$ is a smooth ample divisor on $X$. Suppose that $D$ is of high degree inside $X$, then the map $\CH_i(D)\to \CH_i(X)$ is injective, for a certain range of $i$. For more precise formulation please see \cite{BIL}.

Inspired by the above conjectures this manuscript is an attempt to investigate the nature of the kernel of the map $\CH_2(X_t)\to \CH_2(X)$, where $X$ is smooth projective hypersurface of degree $3$  in $\PR^6$  and $X_t$ is a smooth hyperplane section. Also we investigate the same question for one cycle on a smooth hyperplane section of a hypersurface of degree greater or equal than $7$ in $\PR^5$. The main aim is to understand the non-injectivity locus of the push-forward homomorphism at the level of cycles. The main technique involves monodromy argument on the cohomology of hyperplane sections. The main result is as follows:

\smallskip

\textit{Let $X$ be a smooth hypersurface of degree greater or equal than $7$ in $\PR^5$. Then for a very general hyperplane section $X_t$ of $X$, the kernel of the push-forward map $A_1(X_t)\to A_1(X)$ is parametrised by
$$\{Z\subset X_t|Z\not\in \bcK_t, t\in U\}\subset \bcC_{1,d,U}\;,$$
here $\bcK$ a generic multisection of $\bcC_{1,d,U}$, which is the relative Chow scheme parametrizing one cycles on smooth hyperplane sections of $X$.}

\smallskip

{\small \textbf{Acknowledgements:} The author is indebted to Jaya Iyer for precise formulation of the main question concerned in the paper. The author is thankful to N.Fakhruddin and V.Srinivas for helpful discussion regarding the theme of the manuscript. The author is academically indebted to Claire Voisin whose ideas inspire this work.}

\smallskip

Thorough-out this text we work over the field of complex numbers.

\section{Push-forward homomorphism for a very general hyperplane section of a cubic hypersurface in $\PR^6$}

Let $X$ be a hypersurface of degree $3$ in $\PR^6$. Let $t$ in ${\PR^6}^*$ be such that the corresponding hyperplane section $H_t\cap X$ is smooth. Call it $X_t$. Then we consider the closed embedding $j_t:X_t\to X$. Consider the push-forward homomorphism $j_{t*}$ from $\CH_2(X_t)$ to $\CH_2(X)$. We claim that this homomorphism is injective for a very general $t$. First we notice that $X_t$ is Fano, hence rationally connected. Hence the algebraic equivalence coincides with the homological equivalence on $X_t$ and the group of dimension $2$ algebraically trivial cycles modulo rational equivalence is isomorphic to the corresponding intermediate Jacobian. Also the intermediate jacobian $IJ_2(X_t)$ is isomorphic to $IJ_2(X)$ by the Lefschetz hyperplane theorem. Therefore it reduces to prove that the Neron-Severi group $\NS_2(X_t)$ injects into $\NS_2(X)$ for a very general $t$.

Let $U$ in ${\PR^6}^*$ parametrises the smooth hyperplane sections of $X$.

First we introduce the following notations. Let $\bcH_i$ denote the Hilbert scheme parametrising degree $i$ dimension $2$ subschemes on $X$. Let $\bcH_{i,t}$ denote the closed subscheme of $\bcH_i$, which parametrises the two dimensional Zariski closed subschemes of $X$ incident on $X_t$. More generally
$$\bcH_{i,U}:=\{(t,S)|S\subset X_t, t\in U\}$$
Also we consider the relative Chow scheme
$$\bcC_{i,U}:=\{(t,Z)|\Supp(Z)\subset X_t, t\in U\}$$

Now we prove the following theorem following the ideas of C.Voisin, \cite{Vo}[chapter 3, last section]:

\begin{theorem}
Let $X$ be a hypersurface in $\PR^6$, such that the smooth hyperplane sections $X_t$, satisfies the following:

i) $H^{3,1}(X_t)\cap \ker(H^4(X_t,\CC)\to H^6(X,\CC))$ is non-zero.

ii) The Hodge conjecture is true on $(2,2)$ classes in $H^4(X_t,\CC)$.

Then for a very general $t$ in $U$, the non-injectivty locus of the map from $\NS_2(X_t)$ to $\NS_2(X)$ is parametrised by
$$\{z\subset X_t|z\not\in K_t, t\in U\}\subset \bcC_{2,d,U}$$
where $K$ is a generic multisection of $\bcC_{2,d,U}$.
\end{theorem}

\begin{proof}

We now prove that the projection map from $\bcH_{i,U}$ to $U$ is dominant if the pairs $(t,S)$ in $\bcH_{i,U}$ satisfies the condition:

if the class of $S$ in $\NS_2(X)$ is zero then $S$ is actually zero in $\NS_2(X_t)$. That is if $S$ is homologically trivial on $X$, then it is actually homologically trivial on $X_t$.

So assume that $\bcH_{i,U}\to U$ is dominant. Consider an embedding of $\bcH_{i,U}$ into some projective space. For a general multi-hyperplane-section $K$ of $\bcH_{i,U}$, the map from $K\to U$ is generically finite. Then we throw out the Zariski closed subset of $U$, on which the map $K\to U$ is not smooth. Then we have a smooth, proper map $K\to V$. Hence by Ehressmann's theorem the map $K\to V$ is a fibration of smooth manifolds. Consider $t$ in $V$, let $S_{1,t},\cdots,S_{n,t}$ be the fiber over $t$. So $S_{i,t}$ is incident on $X_t$. Consider the cycle classes of $S_{i,t}$, call them $\alpha_{i,t}$. Consider the local system generated by $\alpha_{i,t}$. This forms a local system because $K\to V$ is fibration. This local system is a sub-local system of $R^4\phi_*\QQ|_V$, where $\phi$ is the map from $\bcX_V=\{(x,t)|x\in X_t, t\in V\}$ to $V$. Call this local system as $F$. Consider the morphism of sheaves
$$F\to R^4\phi_*\QQ|_V\to R^6\phi_*\QQ|_V\;.$$
Consider the kernel of this morphism, call it $F'$. Then $F'$ is a local subsystem in $(R^4\phi_*\QQ_{van})|_V$ which is by definition the local system associated to
$$\ker(H^4(X_t,\QQ)\to H^6(X,\QQ))$$
for $t$ in $V$. Now $V$ is Zariski open in $U$, hence $U\setminus V$ is of real codimension greater or equal than $2$. Therefore $\pi_1(V,t)$ surjects onto $\pi_1(U,t)$ for any $t$ in $V$. The representation of $\pi_1(U,t)$ is irreducible on
$$ker(H^4(X_t,\QQ)\to H^6(X,\QQ))\;.$$
Therefore the corresponding local system $(R^4\phi_*\QQ_{van})|_V$ is indecomposable. $F'$ is a sub-local system of $(R^4\phi_*\QQ_{van})|_V$, therefore $F'=0$ or $F'=(R^4\phi_*\QQ_{van})|_V$.
In the later case since the Hodge conjecture is true for $(2,2)$ classes, we have that all the classes in
$$ker(H^4(X_t,\QQ)\to H^6(X,\QQ))\;,$$
are Hodge classes, hence
$$H^{3,1}(X_t)\cap ker(H^4(X_t,\QQ)\to H^6(X,\QQ))=0 $$
contradicting the assumption of the theorem. Therefore $F'=0$.

Suppose that  the relative Chow scheme $\bcC_{2,d, U}$ consisting of pairs $(t,Z)$ such that support of $Z$ is contained in $X_t$, maps dominantly onto $U$. Here $Z$ is a degree $d$, dimension $2$ cycle on $X$. Then the argument as above produces a local system $F'$ associated to the projection $\bcC_{2,d,U}\to U$. This local system, by monodromy is either all of $R^4\phi_*\QQ_{van}|_V$ or it is zero. But by the assumption of the theorem, the first possibility do not occur. Hence $F'=0$.   Therefore the cycles in $\NS_2(X_t)$, parametrised by a hyperplane multisection $K$ of $\bcC_{2,d,U}$ which are homologically trivial on $\NS_2(X)$ are zero, if the map from $\bcC_{2,d,U}$ to $U$ is dominant. But it could happen that a cycle $z$ is supported on $X_t$, it is homologically trivial on $X$, but need not be in the fiber $K_t$. Therefore the noninjectivity locus of the map $\NS_2(X_t)\to \NS_2(X)$ is supported on the set 
$$\{z\subset X_t|z\not\in K_t, t\in U\}\subset \bcC_{2,d,U}$$
such that the map $\bcC_{2,d,U}\to U$ is dominant.  Consider all $d$ such that $\bcC_{2,d,U}\to U$ is not dominant. Let $Z_d$ be the image of $\bcC_{2,d,U}$ under the above map such that the image is a proper algebraic subset of $U$. Then for $t\in U\setminus \cup_d Z_d$, the cycles in $\NS_2(X_t)$ are parametrized by $\bcC_{2,d,U}$ such that the map from $\bcC_{2,d,U}$ to $U$ is dominant. Therefore by the previous monodromy argument the non-injectivity locus is given by 
$$\{z\subset X_t|z\not\in K_t, t\in U\}\subset \bcC_{2,d,U}$$
where $K$ is a generic hyperplane section of $\bcC_{2,d,U}$.
\end{proof}

\begin{corollary}
\label{cor1}
Let $X$ be a cubic fivefold in $\PR^6$. Then for a very general hyperplane section $X_t$ of $X$, the push-forward at the level of $\CH_2$ has the kernel parametrised by
$$\{z\subset X_t|z\not\in K_t, t\in U\}\subset \bcC_{2,d,U}$$
for $K$ a generic multisection of $\bcC_{2,d,U}$.
\end{corollary}
\begin{proof}
Let $X$ be a hypersurface of degree $3$, in $\PR^6$. Then a smooth hyperplane section of $X$ is Fano, hence rationally connected and satisfies the Hodge Conjecture. On the other we have that since $H^{3,1}(X_t)$ is one dimensional, the condition of the above theorem is satisfied.
\end{proof}

Now it is known (due to Beauville and Donagi) \cite{BD} that for a cubic fourfold $X_t$ that $H^4(X_t,\QQ)$ is isomorphic to $H^2(F(X_t),\QQ)$, where $F(X_t)$ is the Fano variety of lines on $X_t$. The Fano variety of lines on $X$, is denoted by $F(X)$. It is a smooth projective variety of dimension $6$. We are now interested in understanding the kernel of the Neron severi group $\NS^1(F(X_t))\to \NS^2(F(X))$ and we prove the following theorem:

\begin{theorem}
For a very general $t$,  the kernel of the push-forward at the level of  Neron-Severi group $\NS^1(F(X_t))\to \NS^2(F(X))$ is parametrized by
$$\{Z\subset F(X_t)|Z\not\in \bcK_t, t\in U\}\subset \bcC_{3,d,U}$$
where 
$$\bcC_{3,d,U}=\{(Z,t)|Z\subset F(X_t)\}\;.$$
\end{theorem}

\begin{proof}
 Consider the following commutative diagram:

$$
  \diagram
 \NS^2(X_t)\ar[dd]_-{} \ar[rr]^-{} & & \NS^3(X) \ar[dd]^-{} \\ \\
  \NS^1(F(X_t)) \ar[rr]^-{} & & \NS^2(F(X))
  \enddiagram
  $$

Also the commutative diagram at the level of cohomology:

$$
  \diagram
 H^4(X_t,\QQ)\ar[dd]_-{} \ar[rr]^-{} & & H^6(X,\QQ) \ar[dd]^-{} \\ \\
  H^2(F(X_t),\QQ) \ar[rr]^-{} & & H^4(F(X),\QQ)
  \enddiagram
  $$
By the corollary \ref{cor1}, we have that the map from $\NS^2(X_t)$  to $\NS^3(X)$ is supported on 
$$\{z\subset X_t|z\not\in K_t, t\in U\}\subset \bcC_{2,d,U}$$
where $K$ is a generic multisection of $\bcC_{2,d,U}$. Using this information and the fact that the kernel
$$\ker(H^4(X_t,\QQ)\to H^6(X,\QQ))$$
is an irreducible $\pi_1(U,t)$ module ($U$ is the collection of smooth hyperplane sections of $X$), we want to prove that the kernel of the map $\NS^1(F(X_t))\to \NS^2(F(X))$ is supported on
$$\{Z\subset F(X_t)|Z\not\in \bcK_t\}\subset \bcC_{3,U,d}$$
where $\bcK$ is a generic multisection of $\bcC_{3,U,d}$ parametrizing dimension three cycles $Z$ such that $Z$ is supported on $F(X_t)$,  for a very general $X_t$. Note that we have an isomorphism $H^4(X_t,\QQ)$ and $H^2(F(X_t),\QQ)$.  We have the induced action of $\pi_1(U,t)$ on $H^2(F(X_t),\QQ)$ and image of $\ker(H^4(X_t,\QQ)\to H^6(X,\QQ))$ in $H^2(F(X_t),\QQ)$ is irreducible with this action.
Now suppose that we have an element $\alpha$ in $\NS^1(F(X_t))$, which goes to zero under the push-forward from $\NS^1(F(X_t))$ to $\NS^2(F(X))$. Then as previous consider the relative Chow scheme $\bcC_{3,U,d}$ consisting of pairs $(Z,t)$, such that support of $Z$ is contained in $F(X_t)$. Suppose that this relative Chow scheme maps dominantly onto $U$.  Here $Z$ is a dimension $3$ cycle of degree $d$ on $F(X)$. Consider a hyperplane multisection of $\bcC_{3,U,d}$, say $\bcK$ which is smooth and maps generically finitely to $U$. Throwing out a Zariski closed subset of $U$, we have $\bcK\to V$ a proper submersion of smooth  manifolds. Then $\bcK\to V$ gives rise to a fibration. Hence we have a local system consisting of the cohomological cycle classes of  the elements of the fibers of $\bcK\to V$ which is finite. Let us denote this local system by $F=\{\alpha_{1,t},\cdots,\alpha_{i,t}\}$. Then $F$ is a local sub-system of the system $H^2(F(X_t),\QQ)$, which is isomorphic to $H^4(X_t,\QQ)$. Hence consider
$$F'=F\cap \ker(H^4(X_t,\QQ)\to H^6(X,\QQ))\;.$$
Since the representation of $\pi_1(U,t)$ is irreducible on
$$\ker(H^4(X_t,\QQ)\to H^6(X,\QQ))$$
we have that $F'=0$ or $\ker(H^4(X_t,\QQ)\to H^6(X,\QQ))$. That means that the image of the local system given by $\{\alpha_{i,t}\}$ in $H^4(X_t,\QQ)$ is either zero or all of the vanishing cohomology
$$\ker(H^4(X_t,\QQ)\to H^6(X,\QQ))\;.$$
In the second case we have all the elements of vanishing cohomology are given by classes of algebraic subvarieties, which is not true as
$$H^{3,1}(X_t)\cap \ker(H^4(X_t,\QQ)\to H^6(X,\QQ))\neq 0\;.$$
Therefore for a very general $t$, the map from $\NS^1(F(X_t))$ to $\NS^2(F(X))$ has the non-injectivity locus parametrized by
$$\{Z\subset F(X_t)|Z\not\in \bcK_t, t\in U\}\subset \bcC_{3,d,U}$$
for a generic hyperplane section $\bcK$ of $\bcC_{3,d,U}$ parametrizing the cycles of dimension $3$ on $F(X)$, which are supported on $F(X_t)$.
\end{proof}

\section{Hypersurfaces of high degree in $\PR^5$ and their hyperplane section}
Let $X$ be a smooth hypersurface in $\PR^5$, of degree greater than or equal to $7$. Consider a hyperplane sections of $X$, say $X_t$.  By the theorem of Green and Voisin \cite{Vo}[theorem 7.19], \cite{G}, we know that the image of the Abel-Jacobi map from $A_1(X_t)$ (the group of homologically trivial cycles modulo rational equivalence) to the corresponding intermediate Jacobian is torsion. Therefore all elements of $A_1(X_t)$ are torsion by a result of Colliot-Thelene and Sansuc, \cite{CSS}. Now consider the push-forward homomorphism from $A_1(X_t)$ to $A_1(X)$. We would like to investigate the kernel of this map and we prove the following theorem:

\begin{theorem}
Let $X$ be a smooth hypersurface of degree greater or equal than $7$ in $\PR^5$. Then for a very general hyperplane section $X_t$ of $X$, the kernel of the push-forward map $A_1(X_t)\to A_1(X)$ is parametrised by
$$\{Z\subset X_t|Z\not\in \bcK_t, t\in U\}\subset \bcC_{1,d,U}\;,$$
here $\bcK$ a generic multisection of $\bcC_{1,d,U}$.
\end{theorem}

\begin{proof}
Suppose that the map from $\bcC_{1,d,U}$ to $U$ is dominant.
Then by bertini's theorem for a general $t$, there exists a smooth hyperplane section $\bcK$ of $\bcC_{1,d,U}$, such that all the fibers of the projection $\bcK\to U$ are smooth (after throwing out a Zariski closed subset of $U$) and the projection is generically finite.
Let $Z_{1t},\cdots,Z_{nt}$ be the fiber of the projection over $t$. Then we have an action of $\pi_1(U,t)$, on the Abel-Jacobi image of the cycles $Z_{it}$ in the torus $IJ^2(X_t)$(this family of intermediate Jacobians gives a fibration over $u$). Let $V$ denote the sub-torus of $IJ^2(X_t)$, generated by the Abel-Jacobi images of $Z_{it}$. Since all these images of $Z_{it}$'s are torsion, they all belong to $H^3(X_t,\QQ)/H^3(X_t,\ZZ)$. So consider the lifts of Abel-Jacobi images of $Z_{it}$, in $H^3(X_t,\QQ)$. Call the vector space generated by them as $V$. The lifts actually belong to the kernel of the Gysin map from $H^3(X_t,\QQ)$ to $H^5(X,\QQ)$ (because the lift at the level of cohomology with complex coefficients is zero and it is unique by the homotopy lifting property). By the Picard-Lefschetz formula the action of $\pi_1(U,t)$ on this kernel is irreducible. Hence the sub-representation $V$ is either trivial or all of the Gysin kernel. So it means that for a very general  $t$ in $U$, the push-forward from $A_1(X_t)$ to $A_1(X)$ has the kernel parametrised by
$$\{Z\subset X_t|Z\not\in \bcK_t, t\in U\}\subset \bcC_{1,d,U}\;.$$
\end{proof}

\end{document}